\documentclass[11pt]{article}
\usepackage{amsmath,amsthm,amsfonts,amssymb,mathrsfs,bm}
\usepackage{amsmath,amsthm,amsfonts,amssymb,bm,wasysym}
\usepackage{subfigure}
\usepackage{graphicx}
\usepackage[usenames]{color}
\usepackage{verbatim}
\usepackage{hyperref}

\parskip \medskipamount
\newtheorem{theorem}{Theorem}[section]

\numberwithin{equation}{section}
\newtheorem{lemma}[theorem]{Lemma}
\newtheorem{proposition}[theorem]{Proposition}
\newtheorem{corollary}[theorem]{Corollary}

\newtheorem{remark}[theorem]{Remark}

\newtheorem{claim}[theorem]{Claim}

\numberwithin{equation}{section}

\def\Z{\mathbb{Z}}

\def\R{\mathbb{R}}

\def\bP{\mathbb{P}}

\def\F{\mathcal{F}}

\def\B{\mathcal{B}}
\def\LL{\mathcal{L}}

\def\bP{\mathbb{P}}

\def\H{\mathcal{H}}
\renewcommand{\phi}{\varphi}
\renewcommand{\epsilon}{\varepsilon}

\def\K{\mathcal{K}}
\def\GG{\mathcal{G}}

\allowdisplaybreaks

\newcommand{\1}{{\text{\Large $\mathfrak 1$}}}

\newcommand{\til}{\widetilde}

\newcommand{\pr}[1]{\mathbb{P}\!\left(#1\right)}
\newcommand{\E}[1]{\mathbb{E}\!\left[#1\right]}
\newcommand{\estart}[2]{\mathbb{E}_{#2}\!\left[#1\right]}
\newcommand{\prstart}[2]{\mathbb{P}_{#2}\!\left(#1\right)}
\newcommand{\prcond}[3]{\mathbb{P}_{#3}\!\left(#1\;\middle\vert\;#2\right)}
\newcommand{\econd}[2]{\mathbb{E}\!\left[#1\;\middle\vert\;#2\right]}

\newcommand{\pru}[1]{\mathbb{P}_U\!\left(#1\right)}

\newcommand{\econdu}[2]{\mathbb{E}_U\!\left[#1\;\middle\vert\;#2\right]}

\newcommand{\tn}{|\kern-.1em|\kern-0.1em|}

\newcommand\be{\begin{equation}}
\newcommand\ee{\end{equation}}

\def\bP{\mathbb{P}}

\begin{document}
\title{\bf Martingale defocusing and transience of a self-interacting random walk}

\author{
Yuval Peres\thanks{Microsoft Research, Redmond, Washington, USA; peres@microsoft.com} \and Bruno Schapira\thanks{Aix Marseille Universit\'e, CNRS, Centrale Marseille, I2M, UMR 7373, 13453 Marseille, France; bruno.schapira@univ-amu.fr} \and Perla Sousi\thanks{University of Cambridge, Cambridge, UK;   p.sousi@statslab.cam.ac.uk}
}
\maketitle
\begin{abstract}
Suppose that $(X,Y,Z)$ is a random walk in $\Z^3$ that moves in the following way: on the first visit to a vertex only $Z$ changes by $\pm 1$ equally likely, while on later visits to the same vertex $(X,Y)$ performs a two-dimensional random walk step. We show that this walk is transient thus answering a question of Benjamini, Kozma and Schapira. One important ingredient of the proof is a dispersion result for martingales. 
\newline 
MSC 2010 \emph{subject classifications.} Primary 60K35.
\end{abstract}

\section{Introduction}

In this paper we study the following self-interacting random walk $(X,Y,Z)$ in $\Z^3$. On the first visit to a vertex only $Z$ changes by $\pm1$ equally likely, while on later visits to the same vertex $(X,Y)$ performs a two dimensional random walk step, i.e.\ it changes by $(\pm1,0)$ or $(0,\pm1)$ all with equal probability. This walk was conjectured in~\cite{BenKozScha} to be transient. 

This model fits into the wider class of excited random walks which were first introduced by Benjamini and Wilson~\cite{BenWilson}. They study walks that on the first visit to a vertex in~$\Z^d$ have a bias in one direction while on later visits they make a simple random walk step. There has been a lot of active research in this type of model; see the recent survey~\cite{KosZer} and the references therein. 

Another process of this flavour was analysed in~\cite{Serguei}; suppose that $\mu_1, \mu_2$ are two zero-mean measures in $\R^3$ and consider any adapted rule for choosing between $\mu_1$ and $\mu_2$. By adapted rule, we mean that the next choice every time depends on the history of the process up to this time. 
In~\cite{Serguei} it was proved that if each measure is supported on the whole space, then for any adapted rule, the resulting walk in~$\R^3$ is transient. In~\cite{RaimSchap2} transience and recurrence properties and weak laws of large numbers were also proved for specific choices of one-dimensional measures; for instance when $\mu_1$ is the distribution of simple random walk step and $\mu_2$ the symmetric discrete Cauchy law. 

A larger class of such processes are the so-called self-interacting random walks, which are not Markovian, since the next step depends on the whole history of the process up to the present time. For instance the edge or vertex reinforced random walks have attracted a lot of attention, see e.g.~\cite{Omer, BrunoSch, EscTW, MerkRolSil, Pemsurvey, SabTar, Tarres, Toth}.

\begin{theorem}
\label{theotransience}
Let $W_t=(X_t,Y_t,Z_t)$ be a random walk in $\Z^3$ such that on the first visit to a vertex only $Z_t$ changes to $Z_t\pm 1$ equally likely, while on later visits to a vertex $(X_t,Y_t)$ makes a two dimensional simple random walk step. Then $W$ is transient, i.e.\ $\|W_t\|\to\infty$ as $t\to \infty$ almost surely.
\end{theorem}

\begin{figure}[h!]
\begin{center}
\includegraphics[scale=0.5]{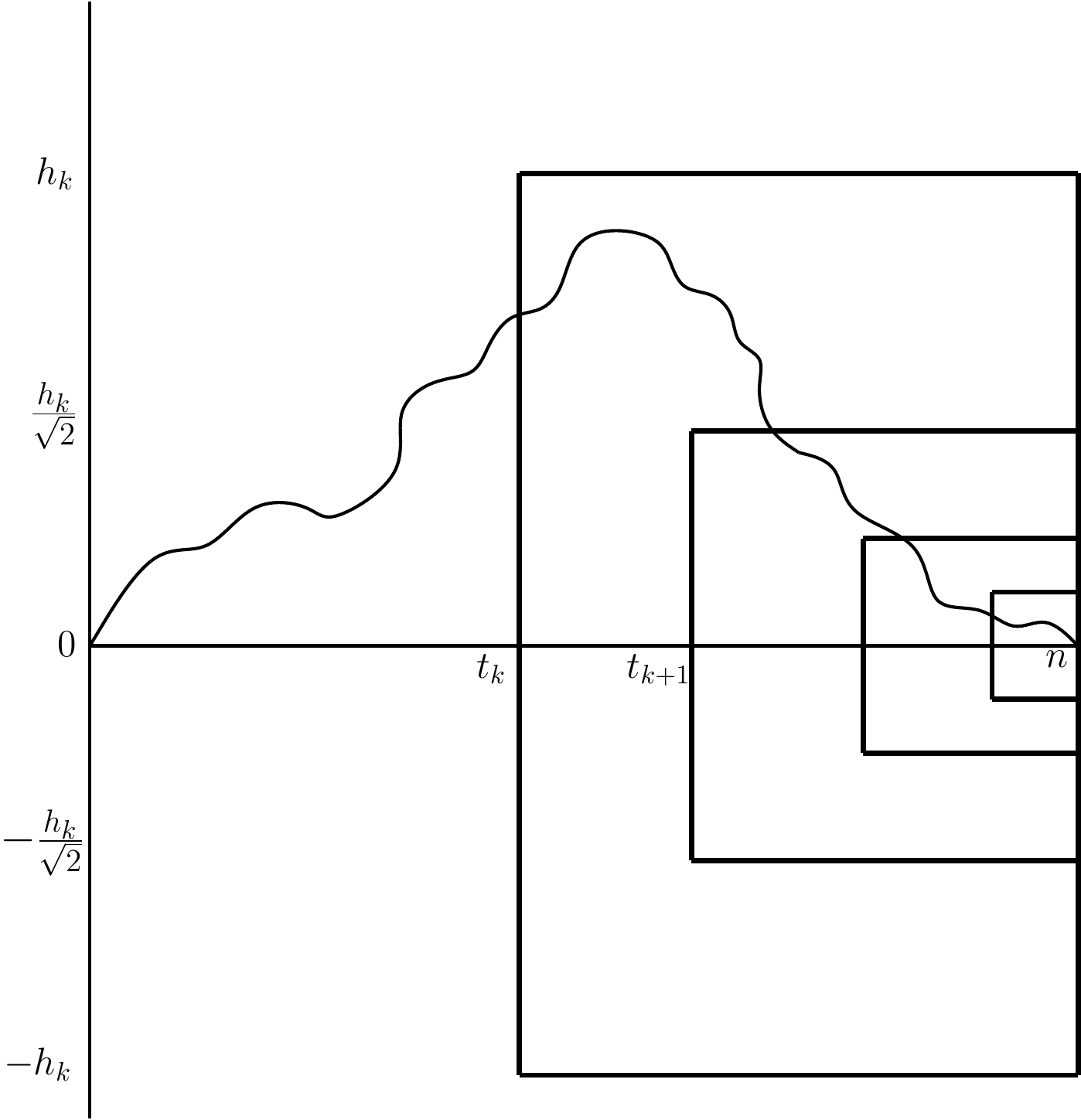}
\caption{The rectangles $[t_k,n]\times[-h_k,h_k]$ and the graph of $M$\label{fig:boxes}}
\end{center}
\end{figure}

We now give a quick overview of the proof of Theorem \ref{theotransience}. By conditioning on all the jumps of the two dimensional process $(Y,Z)$ and looking at the process $X$ only at the times when $(Y,Z)$ moves, we obtain a martingale $M$. Then we need to obtain estimates for the probability that $M$ is at $0$ at time~$n$ so that when multiplied by $1/n$ it should be summable. In Section~\ref{sec:martdef} we state and prove a proposition that gives estimates for a martingale to be at $0$ at time $n$ when it satisfies certain assumptions. We now state a simpler form of this proposition.
\begin{corollary}\label{cor:easycase}
Let $M$ be a martingale satisfying almost surely
$$\econd{(M_{k+1}-M_k)^2}{\F_k}\ge 1 \quad \textrm{and} \quad |M_{k+1}-M_k|\le  (\log n)^a,$$
for all $k\le n$ and some $a<1$. Then there exists a positive constant $c$, such that
\[
\pr{M_n=0} \leq \exp\left(-c(\log n)^{1-a} \right).
\]
\end{corollary}

We remark that related results were recently proved by Alexander in~\cite{Kenneth} and by Armstrong and Zeitouni in~\cite{ArmZei}.

In order to prove Corollary~\ref{cor:easycase} we use the same approach as in~\cite{GurPerZei}. More specifically, we consider the rectangles as in Figure~\ref{fig:boxes}, where the widths decay exponentially and $t_k=n-n/2^k$ for $k< \log_2(n)$. It is clear that $\{M_n=0\}$ only if the graph of $M$ hits all the rectangles. Note that it suffices to show that for most rectangles conditionally on hitting them, the probability that the graph of $M$ does not hit the next one is lower bounded by $c/(\log n)^a$. This is the content of Proposition~\ref{lem:martingale}
in Section~\ref{sec:martdef}. In order to control the probabilities mentioned above, we also have to make sure that the two dimensional process visits enough new vertices in most intervals $[t_k,t_{k+1}]$. This is the content of Proposition~\ref{prop:U} that we state and prove in Section~\ref{sec:proof}.

In Section~\ref{sec:example} we prove the following lemma, which shows that there is no dispersion result in the case $a=1$, with general hypotheses like in Corollary~\ref{cor:easycase}.

\begin{lemma}\label{lem:example}
There exists a positive constant $c$, such that for any $n$, there exists a martingale~$(M_k)_{k\leq n}$ sastisfying almost surely
\[
\econd{(M_{k+1}- M_k)^2}{\F_k} \geq 1 \quad \textrm{and}\quad |M_{k+1}- M_k| \leq \log n,
\]
for all $k\le n$, yet 
\[
\pr{M_n=0} \geq c.
\]
\end{lemma}

{\bf{Notation:}} For functions $f,g$ we write $f(n)\lesssim g(n)$ if there exists a universal constant $C>0$ such that $f(n)\leq Cg(n)$ for all $n$. We write $f(n)\gtrsim g(n)$ if $g(n)\lesssim f(n)$. Finally we write $f(n)\asymp g(n)$ if both $f(n)\lesssim g(n)$ and $f(n)\gtrsim g(n)$. 
We write $\B(x,r)$ to denote the ball in the~$\LL^1$-
 metric centered at $x$ of radius $r$. Note also that in the rest of the paper we use $c$ for a constant whose value may change from line to line.

\section{Martingale defocusing}\label{sec:martdef}

In this section we state and prove a dispersion result for martingales. Then in Section~\ref{sec:proof} we use it to prove our main result, Theorem~\ref{theotransience} when $a=1/2$.  

We call the quadratic variation of a martingale $M$, the process $(V_t)_{t\ge 1}$ defined by 
\[
V_t = \sum_{\ell=1}^{t}\econd{(M_\ell - M_{\ell-1})^2}{\F_{\ell-1}}.
\]

\begin{proposition}\label{lem:martingale}
Let $\rho>0$ be given. 
There exists a positive constant $c$ and $n_0\geq 1$
such that the following holds for any $a\in (0,1)$. 
Suppose that $M$ is a martingale with quadratic variation $V$
and suppose that $(G_k)_k$ is an i.i.d.\ sequence of geometric random variables with mean $2$ satisfying
\begin{align}\label{eq:increment}
|M_{k+1} - M_k| \leq G_k \quad \forall k.
\end{align}
For each $1\leq k< \log_2 (n)$ we let $t_k=n-n/2^k$ and 
\[
A_k = \left\{ V_{t_{k+1}} - V_{t_k} \geq \rho(t_{k+1} - t_k)/(\log n)^{2a}  \right\}.
\]
Suppose that for some $N\ge 1$ and $1\le k_1<\ldots< k_N<  \log_2(n)/2$, it holds
\begin{align}\label{eq:condquad}
\pr{\bigcap_{i=1}^N A_{k_i}}=1.
\end{align}
Then we have for all $n\geq n_0$ 
\[
\pr{M_n=0} \leq  \exp\left(-cN/(\log n)^a \right).
\]
\end{proposition}

\begin{remark}\rm{
We note that the choice of mean $2$ for the geometric random variables in~\eqref{eq:increment} is arbitrary. Any other value would be fine as well.    
}
\end{remark}

\begin{proof}[\bf Proof of Corollary~\ref{cor:easycase}]
If we divide the martingale $M$ by $(\log n)^a$, then it satisfies the hypotheses of Proposition~\ref{lem:martingale} with $N=\log_2 (n)/2$, and hence the statement of the corollary follows.
\end{proof}

Before proving Proposition~\ref{lem:martingale} we state and prove a preliminary result that will be used in the proof later.

\begin{lemma}\label{lem:prelim}
There exists $\rho>0$ such that the following holds. Suppose that $M$ is a martingale satisfying assumption~\eqref{eq:increment} of Proposition~\ref{lem:martingale}.
Let $m<\ell$ and $h>\log(\ell-m+1)$ be given and 
let $\tau=\min\{t\geq m\,:\, |M_t-M_m|\geq h\}$. Suppose that $\pr{V_\ell - V_m \geq h^2/\rho}=1$. Then we have almost surely 
\[
\prcond{\tau\leq\ell}{\F_m}{} \geq \frac{1}{2}.
\]
\end{lemma}

\begin{proof}[\bf Proof]

It is well known that the process $(M_t^2-V_t)$ is a martingale. Since $\tau\wedge 
\ell\geq m$ is a stopping time, by the optional stopping theorem we get
\begin{align}
\label{eq:first}
\econd{M_{\tau\wedge \ell}^2-V_{\tau\wedge \ell}}{\F_m}=M_m^2-V_m. 
\end{align}
Now we claim that 
\begin{align}
\label{eq:second}
\econd{M_{\tau\wedge \ell}^2- M_m^2}{\F_m} \lesssim h^2.
\end{align}
Indeed we can write  
\begin{align*}
\econd{M_{\tau\wedge \ell}^2 - M_m^2}{\F_m}&=\E{(M_{\tau\wedge\ell} - M_{(\tau-1)\wedge\ell} + M_{(\tau-1)\wedge\ell} - M_m)^2} \\&
\leq 2\E{(M_{\tau\wedge\ell} - M_{(\tau-1)\wedge\ell} )^2} + 2h^2,
\end{align*}
where the last inequality follows from the definition of~$\tau$. In order to bound the first term in the right hand side above, we use~\eqref{eq:increment} and the fact that $\tau\geq m$. This way we get
$$|M_{\tau\wedge\ell} - M_{(\tau-1)\wedge\ell}|\le \max_{m\le t\le \ell} |G_t|.$$
So we now obtain 
\begin{align*}
\econd{(M_{\tau\wedge\ell} - M_{(\tau-1)\wedge\ell} )^2}{\F_m} \lesssim \log(\ell-m+1)^2,
\end{align*}
which proves our claim \eqref{eq:second}, using also the hypothesis $h>\log(\ell-m+1)$.
Since by assumption we have $\pr{V_\ell-V_m \geq h^2/\rho}=1$, we obtain that almost surely
$$\econd{V_{\tau\wedge \ell}-V_m}{\F_m} \ge \econd{(V_\ell-V_m)\1(\tau\ge \ell)}{\F_m}\geq \frac{h^2}{\rho}\prcond{\tau\ge \ell}{\F_m}{}.$$
This together with~\eqref{eq:first} and~\eqref{eq:second} and by taking~$\rho$ sufficiently small proves the lemma. 
\end{proof}

We are now ready to give the proof of Proposition~\ref{lem:martingale}. 

\begin{proof}[\bf Proof of Proposition~\ref{lem:martingale}]

We will argue as in~\cite{GurPerZei}, by saying that
in order for $M_n$ to be at $0$, the graph of $M$, i.e.\ the process~$((t,M_t))_{t\le n}$, has to cross the space-time rectangles $H_k$, 
for all $k=1,\dots,\log_2(n)$, which are defined by  
\begin{align}\label{eq:boxes}
H_k:= [t_k,n]\times [-h_k,h_k]  \quad \textrm{with} \quad h_k:= \rho\, \sqrt{\frac{t_{k+1}-t_k}{(\log n)^{2a}}}.
\end{align}
We now define
\[
\sigma_k= \inf\{t\geq t_k: |M_t|\geq h_k\}
\]
For each $k\leq \log_2(n)$ such that $\pr{A_k}=1$ we can apply Lemma~\ref{lem:prelim} with $m=t_k$, $\ell=t_{k+1}$ and $h=2h_k$ if $n$ is sufficiently large so that $h>\log (\ell-m+1)$. We thus deduce that for~$\rho$ sufficiently small and for all $n>n_0$ we have almost surely
\begin{align}\label{eq:sigma}
\prcond{\sigma_k\le t_{k+1}}{\F_{t_k}}{}\ge \frac 12.
\end{align}

Next we claim that a.s.\ conditionally on $\F_{\sigma_k}$ the martingale 
has probability of order $1/(\log n)^a$, to reach level $\pm h_k(\log n)^a/\rho^2$ before returning below level $\pm h_k/\sqrt{2}=\pm h_{k+1}$ (if at least one of these events occurs before time~$n$). 
Indeed assume for instance that $M_{\sigma_k}\ge h_k$. Then the optional stopping theorem shows that on the event $E_k=\{M_{\sigma_k}\geq h_k\}\cap\{\sigma_k\leq n\}$ we have
\begin{align}
\label{tau12}
h_k\le M_{\sigma_k}= \econd{M_{T_1\wedge T_2\wedge n}}{\F_{\sigma_k\wedge n}},
\end{align}
where 
$$T_1:= \inf \{t\ge \sigma_k\ :\  |M_t|\ge h_k (\log n)^a/\rho^2\},$$
and 
$$T_2:=  \inf \{t\ge \sigma_k\ :\  |M_t|\le h_k/\sqrt{2}\}.$$
We deduce from \eqref{tau12} that on $E_k$
$$h_k\le \econd{M_{T_1}\1(T_1<T_2\wedge n)}{\F_{\sigma_k}} + \frac{h_k(\log n)^a}{\rho^2}\,  \prcond{n<T_1\wedge T_2}{\F_{\sigma_k}}{}+\frac{h_k}{\sqrt{2}}.$$ 
Then by using again the bound~\eqref{eq:increment} we get that on $E_k$ 
\begin{eqnarray*}
\econd{M_{T_1}\1(T_1<T_2\wedge n)}{\F_{\sigma_k}}
&\le & \frac{h_k(\log n)^a}{\rho^2} \,  \prcond{T_1<T_2\wedge n}{\F_{\sigma_k}}{} + \E{\max_{t_k\le t\le n} G_t} \\
&\le & \frac{h_k(\log n)^a}{\rho^2} \, \prcond{T_1<T_2\wedge n}{\F_{\sigma_k}}{}+c_1\log(n-t_k+1),
\end{eqnarray*}
where $c_1$ is a positive constant. 
It follows that if $n$ is large enough, then on $E_k$
$$\prcond{T_1\wedge n <T_2}{\F_{\sigma_k}}{} \gtrsim \frac{1}{(\log n)^a}.$$ 
Similarly we get the same inequality with the event $\{M_{\sigma_k}\geq h_k\}$ replaced by $\{M_{\sigma_k}\leq -h_k\}$, and hence we get that almost surely
\begin{align}\label{eq:t1t2}
\prcond{T_1\wedge n <T_2}{\F_{\sigma_k}}{} \1(\sigma_k\leq n) \gtrsim \frac{1}{(\log n)^a}\1(\sigma_k\leq n),
\end{align}
which proves our claim. We now notice that on the event $\{t_k\leq T_1\leq n\}$ we have by Doob's maximal inequality
\begin{align}\label{eq:doob}
\prcond{\sup_{ i\leq n-t_k}|M_{i+T_1} - M_{T_1}|\geq h_k(\log n)^a/(2\rho^2)}{\F_{T_1}}{} \lesssim \frac{n-t_k}{(h_k(\log n)^a/\rho^2)^2} < c_1,
\end{align}
where $c_1$ is a constant that we can take smaller than $1$ by choosing~$\rho$ small enough. Note that we used again~\eqref{eq:increment} in order to bound the $\LL^2$ norm of the increments of the martingale $M$. 

Next we define a sequence of stopping times, which are the hitting times of the space-time rectangles~$(H_k)$ defined in~\eqref{eq:boxes}. More precisely, we let $s_0=0$ and for $i\geq 1$ we let
\[
s_i = \min\{t>s_{i-1}: (t,M_t) \in H_i\}.
\]
Thus for each $k<\log_2(n)$ such that $\pr{A_k}=1$ by using~\eqref{eq:sigma}, \eqref{eq:t1t2} and~\eqref{eq:doob}, we get that on the event $\{s_{k-1}\leq n\}$
\begin{align*}
\prcond{s_k > n}{\F_{s_{k-1}}}{} \gtrsim \frac{1}{(\log n)^a}.
\end{align*}
Using the assumption that the event $\cap_{i=1}^N A_{k_i}$ happens almost surely we obtain
\begin{align*}
\pr{M_n=0} \leq \pr{s_{k_1},\ldots, s_{k_N}\leq n} \leq \left(1-\frac{c
}{(\log n)^a} \right)^N \leq \exp\left( -cN/(\log n)^a \right),
\end{align*}
 for a positive constant $c$, and this concludes the proof.
\end{proof}

\section{Proof of transience}\label{sec:proof}

In this section we prove Theorem~\ref{theotransience}. We first give an equivalent way of viewing the random walk~$W$. Let $\xi_1,\xi_2,\ldots$ be i.i.d.\ random variables taking values~$(0,0,\pm 1)$ equally likely. Let $\zeta_1,\zeta_2,
\ldots$ be i.i.d.\ random variables taking values $(\pm 1,0,0), (0,\pm 1,0)$ all with equal probability, and independent of the $(\xi_i)_i$. Assume that $(W_0,\dots,W_t)$ have been defined, and set  
\[
r_W(t) = \#\{W_0,\dots,W_t\}.
\]
Then
\[
W_{t+1} = \begin{cases}
W_t + \xi_{r_W(t)} &\mbox{if } r_W(t)=r_W(t-1)+1 \\
W_t+\zeta_{t-r_W(t)} &\mbox{otherwise} 
\end{cases}.
\]
To prove Theorem~\ref{theotransience} it will be easier to look at the process at the times when the two dimensional process moves. So we define a clock process $(\tau_k)_{k\ge 0}$ by $\tau_0=0$ and  for $k\ge 0$,
$$\tau_{k+1}=\inf\, \{t>\tau_k \, :\, (X_t,Y_t)\neq (X_{\tau_k},Y_{\tau_k})\}=\inf\, \{t>\tau_k \, :\, t-r_W(t)=k\}.$$
Note that $r_W(0)=1$ and $\tau_k<\infty$ a.s.\ for all~$k$.  
Observe that by definition the process $U_t:=(X_{\tau_t},Y_{\tau_t})$ is a $2d$-simple random walk, and that $r_W(\tau_t) = \tau_t-t+1$. Note that $$Z_t=\sum_{i=1}^{r_W(t)-1}\langle \xi_i,(0,0,1)\rangle.$$
We set $\F_t = \sigma(\xi_1,\ldots, \xi_{\tau_t-t})$, so that $Z_{\tau_t}$ is $\F_t$-measurable for all $t$. 

We call $\mathbb{Q}$ the law of the process $U$.
We denote by $\pru{}$ the law of the process $W$ conditionally on the whole process~$U$, or in other words on the whole sequence $(\zeta_i)_{i\geq 1}$. 
We write~$\mathbb{P} = \mathbb{Q}\times\mathbb{P}_U$ for the law of the overall process $W$.

In the following claim we show that the process $Z$ observed only at the times when the two-dimensional process moves, is a martingale.

\begin{claim}\label{cl:mgs}
Let $M_t = Z_{\tau_t}$. Then $\mathbb{Q}$-a.s.\ we have that $(M_t)$ is an~$(\F_t)$-martingale under $\mathbb{P}_U$.
\end{claim}

\begin{proof}[\bf Proof]
We already noticed that $M_t$ is adapted to $\F_t$ for all $t$. Now since the $\xi_i$'s  are i.i.d.\ and have mean $0$ it follows from Wald's identity that
\begin{align*}
\econdu{Z_{\tau_{t+1}}}{\F_t} = \econdu{Z_{\tau_{t}} + \sum_{i=\tau_t - t+1}^{\tau_{t+1}-t-1}\langle\xi_i,(0,0,1)\rangle}{\F_t} = Z_{\tau_t}, 
\end{align*}
since $\tau_{t+1}-t-1$ is the first time after $\tau_t-t+1$ when we visit an already visited site, and is thus a stopping time.
\end{proof}

\begin{remark}\label{rem:stochdom}\rm{
We note that the jumps of the martingale are stochastically dominated by geometric random variables. More precisely, we can couple the process $M$ (or $W$) with a sequence $(G_t)_{t\ge 0}$ of i.i.d.\ geometric random variables with parameter~$1/2$, such that
\begin{align}
\label{Geom}
|M_{t+1}-M_t|\le G_t \quad \textrm{for all}\, \,t\ge 0. 
\end{align}
}
\end{remark}

Before proceeding, we give some more definitions. 
For $t\ge 0$, set 
$$r_U(t)= \#\{U_0,\dots,U_t\},$$
i.e.\ $r_U(t)$ is the cardinality of the range of the two-dimensional process up to time~$t$. We also set for $t\geq 0$
$$V_t:=\sum_{\ell=1}^t \econdu{(M_\ell-M_{\ell-1})^2}{ \F_{\ell-1}}.$$

\begin{claim}\label{cl:freshsite}
Suppose that $U_\ell$ is a fresh site, i.e.\ $U_\ell\notin \{U_0,U_1,\ldots, U_{\ell-1}\}$. Then
\[
\econdu{(M_{\ell+1} - M_{\ell})^2}{\F_\ell}{} \geq 2.
\]
\end{claim}

\begin{proof}[\bf Proof]
Notice that when $U_\ell$ is a fresh site, then $M_{\ell+1} - M_\ell$ can be written as 
\[
M_{\ell+1} - M_\ell = \sum_{i=1}^{\tau} \lambda_i,
\]
where $(\lambda_i)_i$ are i.i.d.\ random variables taking values $\pm 1$ equally likely and $$\tau=\inf\{ k\geq 2: (\lambda_{k-1},\lambda_{k}) \in \{(-1,+1), (+1,-1)\}\}.$$
Then by the optional stopping theorem we deduce
\[
\econdu{(M_{\ell+1} - M_{\ell})^2}{\F_{\ell}}{} = \E{\left(\sum_{i=1}^{\tau}\lambda_i \right)^2} =\E{\tau} \geq 2,
\]
since $\tau\geq 2$ by definition.
\end{proof}

Before proving Theorem~\ref{theotransience} we state a proposition that we prove later, which combined with the above claim guarantees that the quadratic variation~$V$ of the martingale~$M$ satisfies assumption~\eqref{eq:condquad} of Proposition~\ref{lem:martingale}. The following proposition only concerns the 2d-simple random walk.

\begin{proposition}
\label{prop:U}
For $k\geq 1$ we let $t_k=n-n/2^{k}$ and for $\rho>0$ define
$$\K=\left\{1\le k\le (\log n)^{3/4} \, :\, r_U(t_{k+1})-r_U(t_k) \ge \rho(t_{k+1}-t_k)/\log n\right\}.$$
Then there exist positive constants $\alpha$, $c$ and $\rho_*$ such that for all $\rho<\rho_*$
$$\bP\left(\#\K\le \rho (\log n)^{3/4} \ \Big|\  U_n=0\right) \lesssim \exp(-c(\log n)^\alpha).$$
 \end{proposition}

\begin{proof}[\bf{Proof of Theorem \ref{theotransience}}]
Let $\K$ and $\rho$ be as in Proposition~\ref{prop:U}. Note that $\K$ is completely determined by the 2d-walk. Setting $A=\{\#\K \ge \rho(\log n)^{3/4} \}$ we then have
\begin{align}\label{eq:hit0}
\pr{U_n=M_n=0} &= \E{\1(U_n=0)\pru{M_n=0}\1(A)} + \E{\1(U_n=0)\pru{M_n=0}\1(A^c)}.
\end{align}
On the event $A$, using Claim~\ref{cl:freshsite} we get that there exist $k_1,\ldots, k_{\rho(\log n)^{3/4}} \in \K$ such that 
\[
\pru{\bigcap_{i=1}^{\rho(\log n)^{3/4}} A_{k_i}} = 1,
\]
where the events $(A_i)$ are as defined in Proposition~\ref{lem:martingale}. We can now apply this proposition (with $a=1/2$) to obtain
\begin{align*}
\pru{M_n=0}\1(\#\K \ge \rho(\log n)^{3/4}) \lesssim \exp(-c(\log n)^{1/4}).
\end{align*}
Therefore from~\eqref{eq:hit0} we deduce
\begin{align*}
\pr{U_n=M_n=0} \lesssim \frac{1}{n} \exp(-c(\log n)^{1/4}) + \frac{1}{n}  \exp(-c(\log n)^{\alpha}), 
\end{align*}
where $\alpha$ is as in Proposition~\ref{prop:U}.
Since this last upper bound is summable in $n$, this proves that~$0$ is visited only finitely many times almost surely. Exactly the same argument would work for any other point of $\Z^3$, proving that~$W$ is transient.
\end{proof}

Before proving Proposition~\ref{prop:U} we state and prove a standard preliminary lemma and a corollary that will be used in the proof. 

\begin{lemma}\label{lem:prelim2}
Let $U$ be a simple random walk in~$\Z^2$ starting from~$0$. Then there exists a positive constant~$c$, such that for all $t\le n\log n$ satisfying $\log (n/t) \lesssim (\log n)^{3/4}$ we have
$$\pr{\#\{U_0,\dots,U_n\}\cap  \B(0,\sqrt t)\ge   \frac t{(\log n)^{1/16}}} \lesssim \exp(-c(\log n)^{1/16} ).$$
\end{lemma}
\begin{proof}[\bf Proof]
To prove this we first decompose the path into excursions that the random walk makes across 
$\B(0,2\sqrt t)\setminus \B(0,\sqrt t)$ before time $n$. More precisely define $\sigma_0=0$, and for $i\ge 0$, 
$$\sigma'_i= \inf\{k\ge \sigma_i \ :\ U_k\notin \B(0,2\sqrt t)\},$$ 
$$\sigma_{i+1}=\inf \{k\ge \sigma'_i\ : \ U_k\in \B(0,\sqrt t)\}.$$
Let 
$$N:= \max\{i\ :\ \sigma_i \le n\},$$
be the total number of excursions before time $n$, and for each $i\le N$, let 
$$R_i:= \# \{U_{\sigma_i},\dots,U_{\sigma'_i}\},$$
be the number of points visited during the $i$-th excursion. Of course we have 
\begin{align}
\label{bound:excursions}
\#\{U_0,\dots,U_n\}\cap  \B(0,\sqrt t) \le \sum_{i=1}^N R_i.
\end{align}
Note that every time the random walk is on the boundary of the ball $\B(0,2\sqrt{t})$, it has probability of order $1/(\log(n/t)+4\log \log n)$ to hit the boundary of the ball $\B(0,\sqrt{n}\log n)$ before hitting $\B(0,\sqrt{t})$ (see for instance~\cite{LawLim}). If $T$ is the first exit time from $\B(0,\sqrt{n}(\log n)^2)$, then 
\begin{align}\label{eq:unlike}
\pr{T\leq n} \lesssim e^{-c(\log n)^4},
\end{align}
where $c$ is a positive constant. On the event $\{T\geq n\}$, it is easy to see that $N$ is dominated by a geometric random variable with mean of order $\log(n/t)$. We thus get
\begin{align}\label{eq:enanoumero}
\nonumber&\pr{N\ge  (\log (n/t)+4\log \log n)(\log n)^{1/16}} \\ &\leq \pr{T\leq n} + \pr{N\ge  (\log (n/t)+4\log \log n)(\log n)^{1/16}, T\geq n} \\&
\lesssim \exp\left(-c(\log n)^4\right) + \exp\left(-c(\log n)^{1/16} \right)\lesssim \exp\left(-c(\log n)^{1/16} \right).
\end{align}
Since we have $\estart{\sigma_i' - \sigma_i}{x} \lesssim t$ for all $x\in\B(0,2\sqrt{t})$, by using the Markov property we can deduce 
\begin{align*}
\pr{\sigma'_i-\sigma_i \ge t(\log n)^{1/16}} \le \exp(-c(\log n)^{1/16}).
\end{align*}
Moreover, it follows from \cite[Lemma~4.3]{Bass.Kumagai} and the fact that $\log n\asymp \log t$ that  
$$\pr{\#\{U_{\sigma_i},\dots,U_{\sigma_i+t(\log n)^{1/16}}\}\ge \frac{t}{(\log n)^{7/8} }}
\le \exp(-c(\log n)^{1/16}).$$
Combining the last two inequalities, we get that for any $i$, 
$$\pr{R_i\ge \frac{t}{(\log n)^{7/8}}} \le 2\, \exp(-c(\log n)^{1/16}),$$
where $c$ is a positive constant. Using the assumption that $\log (n/t) \lesssim (\log n)^{3/4}$ together with~\eqref{bound:excursions} and~\eqref{eq:enanoumero} concludes the proof of the lemma.
\end{proof}

\begin{corollary}
\label{cor:range}
Let $U$ be a simple random walk in~$\Z^2$, let $t\leq n$ satisfying $\log (n/(n-t))\lesssim (\log n)^{3/4}$ and let $\varepsilon<1/32$. Then there exists a positive constant $c$ such that 
\[
\prcond{\#\{U_0,\dots,U_{t}\}\cap  \B(0,(\log n)^\varepsilon\sqrt{n-t}) \geq \frac{n-t}{(\log n)^{\frac{1}{16}-2\varepsilon}}}{U_n=0}{} \lesssim  \exp\left(-c(\log n)^{1/16} \right).
\]
\end{corollary}
\begin{proof}[\bf Proof] 
First we use the rough bound: 
\begin{eqnarray*}
\#\{U_0,\dots,U_{t}\}\cap  \B(0,(\log n)^{\varepsilon}\sqrt{n-t}) &\le &  \#\{U_0,\dots,U_{n/2}\}\cap  \B(0,(\log n)^\varepsilon\sqrt{n-t}) \\ 
&& +\ \# \{U_{n/2},\dots,U_n\}\cap  \B(0,(\log n)^\varepsilon\sqrt{n-t}).
\end{eqnarray*}
We now note that if $A$ is an event only depending on the first $n/2$ steps of the random walk, then we have
\begin{align}\label{eq:condition}
\prcond{A}{U_n=0}{} = \frac{\prcond{U_n=0}{A}{} \pr{A}}{\pr{U_n=0}} \lesssim \pr{A},
\end{align}
where the last inequality follows from the local central limit theorem. By reversibility we obtain
\begin{align}\label{eq:cond2}
\nonumber \prcond{\# \{U_{n/2},\dots,U_n\}\cap  \B(0,(\log n)^\varepsilon\sqrt{n-t})\geq \frac{n-t}{2(\log n)^{\frac{1}{16}-2\varepsilon}}}{U_n=0}{} \\
= \prcond{\# \{U_{0},\dots,U_{n/2}\}\cap  \B(0,(\log n)^\varepsilon\sqrt{n-t})\geq \frac{n-t}{2(\log n)^{\frac{1}{16}-2\varepsilon
}}}{U_n=0}{}.
\end{align}
The statement now readily follows by combining Lemma~\ref{lem:prelim2} with~\eqref{eq:condition} and~\eqref{eq:cond2}.
\end{proof}

\begin{proof}[\bf Proof of Proposition~\ref{prop:U}]
Let us consider the events  
$$A_k:=\left\{r_U(t_{k+1})-r_U(t_k) \ge \rho\, \frac {t_{k+1}-t_k}{\log n}\right\},$$
with $\rho>0$ some constant to be fixed later.
Let also $\varepsilon<1/48$, 
$$B_k=\left\{\#\{U_0,\dots,U_{t_k}\} \cap \B(0,(\log n)^\varepsilon\sqrt{n-t_k}) \le \frac{t_{k+1}-t_k}{(\log n)^{\frac{1}{16}-2\varepsilon}}\right\},$$
and 
$$\widetilde B_k:=B_k\cap \left\{U_{t_k}\in \B(0,\sqrt{n-t_k})\right\}.$$
Set for $k=1,\dots,(\log n)^{3/4}$, 
$$\GG_k=\sigma(U_0,\dots,U_{t_k}),$$ 
and note that $\widetilde B_k\in \GG_k$. 
\begin{claim}\label{cl:beginproof}
For any~$k\le (\log n)^{3/4}$ we have almost surely
\begin{align*}
\bP(A_k^c\mid \GG_k)\, \1(\widetilde B_k) \lesssim \, \frac{1}{(\log n)^{\epsilon}}\, \1(\widetilde B_k).
\end{align*}
\end{claim}
\begin{proof}[\bf Proof]
To prove the claim we use two facts. 
On the one hand it follows from \cite[Theorem 1.5]{BCR} that if $\rho$ is small enough, then a.s.  
\begin{align*}
\bP\left(\#\{U_{t_k+1},\dots,U_{t_{k+1}}\}\le 2\rho\, \frac{t_{k+1}-t_k}{\log n}\ \Big|\  \GG_k\right)\le \exp\left(-c(\log n)^{1/6}\right).
\end{align*}
Moreover, on the event $\{U_{t_k} \in \B(0,\sqrt{n-t_k})\}$, 
with probability at most $\exp(-c(\log n)^{2\epsilon})$ the random walk exits the ball $\B(0,(\log n)^{\epsilon}\sqrt{n-t_k})$ before time $t_{k+1}$. Therefore we obtain on the event~$\{U_{t_k} \in \B(0,\sqrt{n-t_k})\}$ that
\begin{align}
\label{Ak:step1}
\bP\left(\#\{U_{t_k+1},\dots,U_{t_{k+1}}\}\cap \B(0,(\log n)^\varepsilon\sqrt{n-t_k}) \le 2\rho\, \frac{t_{k+1}-t_k}{\log n}\ \Big|\  \GG_k\right)\lesssim \exp\left(-c(\log n)^{2\varepsilon}\right).
\end{align}
Suppose now on the other hand that a point is at distance at least~$r=O(\sqrt{t})$ from $U_{t_k}$. Then it is well known (see for instance~\cite{LawLim}) that 
the probability that the walk hits it 
during the time interval $[t_k,t_k +t]$ is~$O(\log(\sqrt{t}/r)/\log \sqrt{t})$. 
If we apply this with $t=t_{k+1}-t_k$, $r=\sqrt{n-t_k}/\log n$, 
and use that~$\#\B(0,r)\asymp r^2$, we deduce that 
\begin{eqnarray}
\label{Ak:step2}
\nonumber && \econd{\#\{U_0,\dots,U_{t_k}\} \cap \{U_{t_k+1},\dots,U_{t_{k+1}}\}\cap \B(0,(\log n)^\varepsilon\sqrt{n-t_k})}{\GG_k}\1(\widetilde B_k) \\
\nonumber &\lesssim &\left(\frac{t_{k+1}-t_k}{(\log n)^2} +  \frac{(t_{k+1}-t_k)\log \log n}{(\log n)^{17/16-2\epsilon}}\right)\, \1(\widetilde B_k)  \\ 
&\lesssim &\frac{t_{k+1}-t_k}{(\log n)^{1+\epsilon}} \, \1(\widetilde B_k).
\end{eqnarray}
We now have almost surely
\begin{align*}
\prcond{A_k}{\GG_k}{}\1(\til{B}_k) \geq \prcond{\#\{U_{t_k+1},\dots,U_{t_{k+1}}\}\cap \B(0,(\log n)^\varepsilon\sqrt{n-t_k}) > 2\rho\, \frac{t_{k+1}-t_k}{\log n}}{\GG_k}{} \1(\til{B}_k) \\-
\prcond{\#\{U_0,\dots,U_{t_k}\} \cap \{U_{t_k+1},\dots,U_{t_{k+1}}\}\cap \B(0,(\log n)^\varepsilon\sqrt{n-t_k})\geq \rho \frac{t_{k+1} -t_k}{\log n}}{\GG_k}{}\1(\til{B}_k) \\
\geq \left( 1 - \exp\left( -c(\log n)^{1/6}\right) - \frac{c_1}{(\log n)^\epsilon}\right) \1(\til{B}_k),
\end{align*}
where the last inequality follows from 
\eqref{Ak:step1}, \eqref{Ak:step2} and Markov's inequality.
\end{proof}

Next, if we write $Q(\cdot)=\prcond{\cdot}{U_n=0}{}$ for the Doob transform of $U$, then $(U_k)_{k\leq n}$ is a Markov chain under~$Q$. 
We let $\til{A}_k =A_k\cap \{U_{t_{k+1}} \in \B(0,2\sqrt{n-t_{k}})\}$. Then we have almost surely
\begin{align}\label{AkBk+}
Q(A_k\mid \GG_k) \1(\til{B}_k) \geq Q\left(\til{A}_k\;\middle\vert \;\GG_k\right) \1(\til{B}_k) \gtrsim \prcond{\til{A}_k}{\GG_k}{}\1(\til{B}_k)\geq p\1(\til{B}_k),
\end{align}
where the penultimate inequality follows by the local central limit theorem as in~\eqref{eq:condition} and the last inequality from Claim~\ref{cl:beginproof} and the fact that 
\[
\prcond{U_{t_{k+1}} \in \B(0,2\sqrt{n-t_{k}})}{U_{t_k} \in \B(0,\sqrt{n-t_k})}{} \geq c>0.
\]
Then we introduce the process $(M_k)_{k\le (\log n)^{3/4}}$, defined by $M_1=0$ and for $k\ge 2$, 
$$M_k:= \sum_{\ell=1}^{k-1} \, \left\{\1(A_\ell\cap \widetilde B_\ell) - Q(A_\ell\mid \GG_\ell)\, \1(\widetilde B_\ell)\right\}.$$
Note that by construction it is a $(\GG_k)$-martingale, under the measure~$Q$.  
 Since the increments of this martingale are bounded, it follows from Azuma-Hoeffding's inequality that 
for any $\kappa>0$, there exists a positive constant $c$ such that  
\begin{align}
\label{bound:mart}
\prcond{|M_{(\log n)^{3/4}}| \ge \kappa\, (\log n)^{3/4}}{ U_n=0}{}\lesssim \exp(-c (\log n)^{3/4}).
\end{align}
As a consequence of Corollary~\ref{cor:range} we get that
\begin{align}
\label{allBk}
1-\prcond{\cap_{k\le (\log n)^{3/4}}\, B_k}{U_n=0}{}\lesssim \exp(-c(\log n)^{1/16} ).
\end{align}

\begin{claim}\label{cl:endofproof}
There exists a positive constant $c$ such that 
\begin{align*}
\prcond{\sum_{k=1}^{(\log n)^{3/4}}\, \1(U_{t_k}\in \B(0,\sqrt{n-t_k})) \le c(\log n)^{3/4}}{U_n=0}{}
\le \exp(-c(\log n)^{3/4} ).
\end{align*}
\end{claim}

\begin{proof}[\bf Proof]
By using reversibility and the local central limit theorem again, it suffices in fact to show the result without conditioning on $U_n=0$, 
and replacing the times $t_k$ by $n-t_k$.  In other words, it suffices to prove that 
\begin{align}
\label{finalgoal}
\pr{\sum_{k=1}^{(\log n)^{3/4}}\, \1(U_{2^k}\in \B(0,2^{k/2})) \le c(\log n)^{3/4}}
\le \exp(-c(\log n)^{3/4} ),
\end{align}
for some $c>0$. This is standard, but for the sake of completeness we give a short proof now. 
We will prove in fact a stronger statement. Call 
$$v_k := \inf \{t\ge 0\ :\ U_t\notin \  \B(0,2^{k/2})\}.$$
Obviously it is sufficient to prove \eqref{finalgoal} with the events 
$\{v_k> 2^k\}$ in place of $\{U_{2^k}\in \B(0,2^{k/2}) \}$.
Set $\H_k=\sigma(U_0,\dots,U_{v_k})$. Then it is well known that we can find a constant $\alpha>0$, 
such that a.s. for any $k$,  
$$\bP(v_{k+1}> 2^{k+1}\mid \H_k)\ge \alpha.$$ 
Then by considering the martingale  
$$M'_k:=\sum_{\ell=1}^k \left(\1(v_\ell > 2^\ell)-\bP(v_\ell>2^\ell\mid \H_{\ell-1})\right),$$
and using the Azuma-Hoeffding inequality the desired estimate follows.
So the proof of the claim is complete.
\end{proof}

By taking $\rho$ and $\kappa$ sufficiently small and using~\eqref{AkBk+}, \eqref{bound:mart}, \eqref{allBk} and Claim~\ref{cl:endofproof} finishes the proof of the proposition.
\end{proof}


\section{Example}\label{sec:example}

In this section we construct the martingale of Lemma~\ref{lem:example}. Before doing so, we recall a self-interacting random walk~$(X,Y,Z)$ in~$\Z^3$ which was mentioned in~\cite{BenKozScha} and is closely related to the random walk of Theorem~\ref{theotransience}; on the first visit to a vertex only~$(X,Y)$ performs a two-dimensional step, while on later visits to the same vertex only~$Z$ changes by $\pm1$ equally likely. Our proof in this case does not apply, or at least another argument is required. Indeed, by looking again at the process $Z$ at the times when $(X,Y)$ moves, we still obtain a martingale, but we do not have a good control on the jumps of this martingale. In particular, up to time $n$, they could be of size of order~$\log n$, which might be a problem as Lemma~\ref{lem:example} shows.

\begin{proof}[\bf Proof of Lemma~\ref{lem:example}]

Define $M_0=0$. Let $(S_k^i)_{k,i}$ be independent (over~$i$) simple random walks on~$\Z$ and let~$(\til{S}_k^i)_{k,i}$ be independent (over~$i$) random walks with jumps that take values~$\pm [\log n]$ equally likely and start from~$0$. 
Let $k_*$ be the first integer such that $n/2^{k_*}\leq (\log n)^2$. We now let
\[
M_k = S_{k}^1 \quad \text{for} \quad k\leq n/2.
\]
We define $t_1$ by
\[
n-t_1 = \frac{n}{2} + \inf\left\{t\geq 0: \left|M_{n/2} +\til{S}_t^1 \right|\leq \log n\right\}.
\]
If $t_1\geq 0$, then we let 
\[
M_{k+n/2} = M_{n/2} + \til{S}_k^1 \quad \text{for}\quad 0\leq k\leq \frac{n}{2} - t_1.
\]
If $t_1<0$, then we let 
\[
M_{k+n/2} = M_{n/2} + \til{S}_k^1 \quad \text{for}\quad 0\leq k\leq \frac{n}{2}.
\]
Suppose that we have defined $t_{\ell}>0$, we now define $t_{\ell+1}$ inductively. We let
\[
M_{k+n-t_\ell} = M_{n-t_\ell} + S_{k}^{\ell+1} \quad \text{for}\quad 0\leq k\leq \frac{t_\ell}{2}
\]
and we also define $t_{\ell+1}$ by
\[
n-t_{\ell+1} = n-\frac{t_\ell}{2} + \inf\left\{t\geq 0: \left|M_{n-t_\ell/2} +\til{S}_t^{\ell+1}  \right|\leq 
\log n  \right\}.
\] 
If $t_{\ell+1}\geq 0$, then we let 
\[
M_{k+n-t_\ell/2} = M_{n-t_\ell/2} + \til{S}_k^{\ell+1} \quad \text{for}\quad 0\leq k\leq \frac{t_\ell}{2} - t_{\ell+1}.
\]
If $t_{\ell+1}<0$, then we let 
\[
M_{k+n-t_\ell/2} = M_{n-t_\ell/2} + \til{S}_k^{\ell+1} \quad \text{for}\quad 0\leq k\leq \frac{t_\ell}{2}.
\]
In this way we define the times $t_\ell$ for all $\ell\leq k^*$, unless there exists $\ell$ such that $t_\ell<0$, in which case we set $t_{m} = 0$ for all $\ell+1\leq m \leq k^*$. If $t_{k^*}>1$, then at time~$n-t_{k^*}+1$ if $d(M_{n-t_{k^*}},0)\neq 0$, then the martingale makes a jump of size~$\pm d(M_{n-t_{k^*}},0)$ equally likely. If $d(M_{n-t_{k^*}},0)=0$, then with probability  $1/(\log n)^2$ it jumps to $\pm [\log n]$, while with probability~$1-1/(\log n)^2$ it stays at~$0$. 
From time $n-t_{k^*}+2$ until time $n$ at every step with probability $1/(\log n)^2$ it jumps to $\pm [\log n]$, while with probability~$1-1/(\log n)^2$ it stays at its current location. 

By the definition of the martingale it follows that it satisfies the assumptions of the lemma. It only remains to check that there exists a positive constant $c$ such that 
$\pr{M_n=0}>c$. We define the events
\begin{align*}
E  = \{M_{n-t_{k^*}+1 =0} \} \,\,\text{and}\,\,
E' = \{ M_{\ell} = 0, \, \text{for all} \, \ell \in \{ n-t_{k^*}+2,\ldots, n\}  \}.
\end{align*}
We now have
\begin{align}\label{eq:firstmn}
\begin{split}
\pr{M_n=0} \geq &\pr{t_1>0,\ldots, t_{k^*}>0, M_{n-t_{k^*}} =0, E, E'} \\&+ \pr{t_1>0,\ldots, t_{k^*}>0, M_{n-t_{k^*}} \neq 0, E, E'}.
\end{split}
\end{align}
By the definition of the times $t_i$, it follows that $t_{i+1}\leq t_i/2$, and hence we deduce that $t_i\leq n/2^i$, which implies that $t_{k^*}\leq n/2^{k^*} \leq (\log n)^2$. We now obtain 
\begin{align}
\label{eq:split}
\begin{split}
&\prcond{E,E'}{t_1>0,\ldots, t_{k^*}>0, M_{n-t_{k^*}}\neq 0}{} \gtrsim\left(1- \frac{1}{(\log n)^2}\right)^{(\log n)^2} \\
& \prcond{E,E'}{t_1>0,\ldots, t_{k^*}>0, M_{n-t_{k^*}}=0}{} \geq \left(1- \frac{1}{(\log n)^2} \right)^{(\log n)^2}.
\end{split}
\end{align} 
Using the estimate for a simple random walk that if $h>0$, then
\[
\prstart{S_k>0, \,\forall k\leq n}{h} \lesssim \frac{h}{\sqrt{n}},
\]
we get for a positive constant $c_1$ that
\begin{align*}
\prcond{t_{\ell+1}>0}{t_\ell>0}{}  
& = 1-\prcond{\inf\{t\geq 0: |M_{n-t_\ell/2} + \til{S}_{t}^{\ell+1}| \leq \log n \} \leq \frac{t_\ell}{2}}{t_\ell>0}{} \\
&\geq  1- \frac{c_1}{\log n}. 
\end{align*}
Hence from~\eqref{eq:firstmn} and~\eqref{eq:split}  together with the above estimate and the fact that $k^*\asymp \log n$, we finally conclude
\begin{align*}
\pr{M_n=0} \gtrsim \left(1-\frac{c_1}{\log n} \right)^{c_2\log n} \cdot \left(1 - \frac{1}{(\log n)^2} \right)^{(\log n)^2} \geq c_3>0
\end{align*}
and this finishes the proof of the lemma. 
\end{proof}

{\textbf{Acknowledgements}}
We thank Kenneth Alexander for useful discussions. B.S. and P.S. also  thank the Theory Group of Microsoft Research for its kind hospitality where part of this work was completed.

\bibliographystyle{plain}
\bibliography{biblio}

\end{document}